\documentclass{amsart}%
\usepackage{amssymb}
\usepackage{amsmath}
\usepackage{latexsym}
\usepackage{hyperref}
\usepackage{amsfonts}
\usepackage{graphicx}%
\newtheorem{theorem}{Theorem}
\newtheorem{lemma}[theorem]{Lemma}

\newcommand{\be}{\begin{equation}}
\newcommand{\ee}{\end{equation}}
\newcommand{\bea}{\begin{eqnarray}}
\newcommand{\eea}{\end{eqnarray}}

\begin{document}
\title{The linear trace Harnack quadratic on a steady gradient Ricci soliton
satisfies the heat equation}
\author{Bennett Chow}
\author{Peng Lu$^{1}$}
\maketitle

All quantities shall be assumed $C^{\infty}$.\smallskip

\textbf{1. Evolution of the linear trace Harnack on steady and shrinking
gradient Ricci solitons\smallskip}

Suppose $\left(  \mathcal{M}^{n},g_{ij}\left(  t\right)  ,h_{ij}\left(
t\right)  \right)  $ satisfies the linearized Ricci flow $\frac{\partial
}{\partial t}g_{ij}=-2R_{ij}$ and $\frac{\partial}{\partial t}h_{ij}=\left(
\Delta_{L}h\right)  _{ij}$, where $\Delta_{L}$ denotes the Lichnerowicz
Laplacian.\footnotetext[1]{Addresses. Bennett Chow: Math. Dept., UC San Diego;
Peng Lu: Math. Dept., U of Oregon.} Define the matrix Harnack quantities (see
Hamilton \cite{H2})%
\[
M_{pq}=\Delta R_{pq}-\frac{1}{2}\nabla_{p}\nabla_{q}R+2R_{pijq}R^{ij}%
-R_{pk}R_{q}^{k},\quad P_{ipq}=\nabla_{i}R_{pq}-\nabla_{p}R_{qi}.
\]
Recall that if $X\left(  t\right)  $ is any time-dependent vector field on
$\mathcal{M}$, then the corresponding linear trace Harnack quantity $Z\left(
h,X\right)  =\operatorname{div}$$\left(  \text{$\operatorname{div}$}\left(
h\right)  \right)  +\left\langle \operatorname{Rc},h\right\rangle
+2\left\langle \text{$\operatorname{div}$}\left(  h\right)  ,X\right\rangle
+h\left(  X,X\right)  $ satisfies (see \cite{CH})%
\begin{align}
\left(  \frac{\partial}{\partial t}-\Delta\right)  Z &  =2h^{pq}\left(
M_{pq}+2P_{ipq}X^{i}+R_{pijq}X^{i}X^{j}\right)  \label{Z evolution for X}\\
&  \quad\;-4\left(  \nabla_{j}X^{i}-R_{j}^{i}\right)  \nabla^{j}\left(
\text{$\operatorname{div}$}(h)_{i}+h_{ik}X^{k}\right)  \nonumber\\
&  \quad\;+2\left(  \text{$\operatorname{div}$}\left(  h\right)  _{j}%
+h_{ij}X^{i}\right)  \left(  \frac{\partial X^{j}}{\partial t}-\Delta
X^{j}-R_{k}^{j}X^{k}\right)  \nonumber\\
&  \quad\;+2h_{ij}\left(  \nabla_{p}X^{i}-R_{p}^{i}\right)  \left(  \nabla
^{p}X^{j}-R^{pj}\right)  .\nonumber
\end{align}

Since $\frac{\partial}{\partial t}\operatorname{Rc}=\Delta_{L}%
\operatorname{Rc}$, we may take $h=\operatorname{Rc}$. Then $2Z\left(
\operatorname{Rc},X\right)  =\Delta R+2\left\vert \operatorname{Rc}\right\vert
^{2}+2\left\langle \nabla R,X\right\rangle +2\operatorname{Rc}\left(
X,X\right)  $, which is Hamilton's trace Harnack quadratic. Furthermore, if
$\operatorname{Rc}+\nabla\nabla f=0$, then $Z\left(  \operatorname{Rc},-\nabla
f\right)  =0$ by $\Delta R+2\left\vert \operatorname{Rc}\right\vert
^{2}=\langle\nabla R,\nabla f\rangle$ and $2R_{ij}\nabla_{j}f=\nabla_{i}R$.

Now consider the linear trace Harnack quantity for the linearized Ricci flow
on a steady gradient Ricci soliton.

\begin{lemma}
\label{Lemon 1}If $\left(  \mathcal{M}^{n},g\left(  t\right)  ,f\left(
t\right)  ,h\left(  t\right)  \right)  $ satisfies $\frac{\partial}{\partial
t}g=-2\operatorname{Rc}=2\nabla\nabla f$, $\frac{\partial f}{\partial
t}=\Delta f$, $\frac{\partial}{\partial t}h=\Delta_{L}h$, then $Z\left(
h,-\nabla f\right)  =\operatorname{div}$$\left(  \text{$\operatorname{div}$%
}\left(  h\right)  \right)  +\left\langle \operatorname{Rc},h\right\rangle
-2\operatorname{div}$$\left(  h\right)  \left(  \nabla f\right)  +h\left(
\nabla f,\nabla f\right)  $ satisfies $\frac{\partial Z}{\partial t}=\Delta
Z$.\footnote[2]{The same is true if we replace $\frac{\partial f}{\partial
t}=\Delta f$ by $\frac{\partial f}{\partial t}=\left\vert \nabla f\right\vert
^{2}$ essentially since $R=-\Delta f=1-\left\vert \nabla f\right\vert ^{2}$.}
I.e., the linear trace Harnack quantity solves the heat equation.
\end{lemma}

\begin{proof}
By $M_{pq}=P_{ipq}\nabla^{i}f$, $P_{ipq}=R_{pijq}\nabla^{j}f$, $\nabla
_{j}\nabla^{i}f+R_{j}^{i}=0$, and $\left(  (\frac{\partial}{\partial t}
-\Delta)\nabla f\right)  ^{j}=R_{k}^{j}\nabla^{k}f$, for $Z\left( h, - \nabla
f\right)  $ each of the terms on the \textsc{rhs} of (\ref{Z evolution for X})
(with $X= - \nabla f$) are zero.
\end{proof}

\textbf{Remark.} (1) The quantity $Z(h,-\nabla f)$ arises naturally from the
following consideration. If $g_{ij}\left(  s\right)  ,f\left(  s\right)  $ are
such that $\frac{\partial}{\partial s}g_{ij}=h_{ij}$ and $\frac{\partial
}{\partial s}f=\frac{H}{2}$, where $H=g^{ij}h_{ij}$ (so that $\frac{\partial
}{\partial s}\left(  e^{-f}d\mu\right)  =0$; see \cite{Per1}), then Perelman's
scalar curvature satisfies%
\[
\frac{\partial}{\partial s}\left(  R+2\Delta f-\left\vert \nabla f\right\vert
^{2}\right)  =Z\left(  h,-\nabla f\right)  -2\left\langle h,\operatorname{Rc}%
+\nabla\nabla f\right\rangle .
\]
Note for a steady gradient Ricci soliton, $\frac{\partial}{\partial s}\left(
R+2\Delta f-\left\vert \nabla f\right\vert ^{2}\right)  =Z\left(  h,-\nabla
f\right)  $ while $\frac{\partial}{\partial t}\left(  e^{-f}d\mu\right)
=\left(  -\Delta f-R\right)  e^{-f}d\mu=0$.

(2) If $g\left(  t\right)  ,f\left(  t\right)  $ solves $\frac{\partial
g}{\partial t}=-2\operatorname{Rc}$ and $\frac{\partial f}{\partial t}=-\Delta
f+\left\vert \nabla f\right\vert ^{2}-R$, then $V\doteqdot(2\Delta f-|\nabla
f|^{2}+R)e^{-f}$ satisfies $\square^{\ast}V=-2|R_{ij}+\nabla_{i}\nabla
_{j}f|^{2}e^{-f}$, where $\square^{\ast}=-\frac{\partial}{\partial t}%
-\Delta+R$ (see \cite{Per1}).\smallskip

The shrinker analogue of Lemma \ref{Lemon 1} is the following.

\begin{lemma}
If $\left(  \mathcal{M}^{n},g\left(  t\right)  ,f\left(  t\right)  \right)  $,
$t<0$, satisfies $\frac{\partial}{\partial t}g=-2\operatorname{Rc}%
=2\nabla\nabla f+\frac{1}{t}g$, $\frac{\partial f}{\partial t}=\left\vert
\nabla f\right\vert ^{2}=\Delta f-\frac{1}{t}f$, and $\frac{\partial}{\partial
t}h=\Delta_{L}h$, then $\left(  \frac{\partial}{\partial t}-\Delta\right)
\left(  t^{2}\left(  Z\left( h, -\nabla f\right)  +\frac{H}{2t}\right)
\right)  =0$.
\end{lemma}

\begin{proof}
Applying $M_{pq}+\frac{1}{2t}R_{pq}=P_{ipq}\nabla^{i}f$, $P_{ipq}%
=R_{pijq}\nabla^{j}f$, $R_{ij}+\nabla_{i}\nabla_{j}f=-\frac{1}{2t}g_{ij}$, and
$\left(  (\frac{\partial}{\partial t}-\Delta)\nabla f\right)  ^{j}-R_{k}%
^{j}\nabla^{k}f=-\frac{1}{t}\nabla^{j}f$ to (\ref{Z evolution for X}) yields
$\left(  \frac{\partial}{\partial t}-\Delta\right)  Z\left(  h,-\nabla
f\right)  =-\frac{2}{t}Z\left(  h,-\nabla f\right)  -\frac{1}{t}\left\langle
h,\operatorname{Rc}\right\rangle -\frac{H}{2t^{2}}$. Because $\left(
\frac{\partial}{\partial t}-\Delta\right)  H=2\left\langle h,\operatorname{Rc}%
\right\rangle $, we conclude%
\[
\left(  \frac{\partial}{\partial t}-\Delta\right)  \left(  Z\left(  h,-\nabla
f\right)  +\frac{H}{2t}\right)  =-\frac{2}{t}\left(  Z\left(  h,-\nabla
f\right)  +\frac{H}{2t}\right)  .
\]

\end{proof}

\textbf{Remark.} If $h=\operatorname{Rc}$, then for a shrinker, since
$\frac{1}{2}\Delta R+\left\vert \operatorname{Rc}\right\vert ^{2}=\frac{1}%
{2}\left\langle \nabla R,\nabla f\right\rangle -\frac{R}{2t}$ and $H=R$, we
have that $Z\left(  h,-\nabla f\right)  +\frac{H}{2t}=Z\left(
\operatorname{Rc},-\nabla f\right)  +\frac{R}{2t}=0$.\smallskip

\textbf{2. Interpolating between Perelman's and Cao--Hamilton's Harnacks on a
steady soliton}\smallskip

Now consider the system $\frac{\partial}{\partial t}g_{ij}=-2R_{ij}%
=2\nabla_{i}\nabla_{j}f$, $\frac{\partial f}{\partial t}=\left\vert \nabla
f\right\vert ^{2}$, and $\frac{\partial u}{\partial t}=\Delta u+Ru$, $u>0$
(when $n=2$ this is the linearized Ricci flow).\footnote[3]{Note that
$\frac{\partial}{\partial t}e^{f}=\Delta e^{f}+Re^{f}$. The trivial case of
the following calculations is $v=f$.} Mimicking Li, Yau, and Hamilton, define
$v=\log u$, $Q=\Delta v+R$, and $L=\frac{1}{2}(\frac{\partial}{\partial
t}-\Delta)-\nabla v\cdot\nabla$. Then%
\begin{equation}
LQ=\left\vert \nabla\nabla v\right\vert ^{2}+\left\langle \operatorname{Rc}%
,\nabla\nabla v\right\rangle +\operatorname{Rc}(\nabla(v-f),\nabla(v-f)).
\label{ElleQueue}%
\end{equation}
Following X.\thinspace Cao and Hamilton \cite{XCaoHamilton}, define the
Harnack $P=2Q+\left\vert \nabla v\right\vert ^{2}+R$. Since%
\[
L\left(  \left\vert \nabla v\right\vert ^{2}+R\right)  =\left\vert
\operatorname{Rc}\right\vert ^{2}-\left\vert \nabla\nabla v\right\vert ^{2},
\]
we have that $P=2\Delta v+\left\vert \nabla v\right\vert ^{2}+3R$ satisfies
the parabolic Bochner-type formula\footnote[4]{Note that $\nabla\nabla
v+\operatorname{Rc}=\nabla\nabla\left(  v-f\right)  $.}%
\[
LP=\left\vert \nabla\nabla v+\operatorname{Rc}\right\vert ^{2}%
+2\operatorname{Rc}(\nabla(v-f),\nabla(v-f)).
\]
In particular, if $\operatorname{Rc}\geq0$, then $LP\geq\frac{1}{n}Q^{2}\geq0$.

More generally, if $\frac{\partial u}{\partial t}=\varepsilon^{-1}\Delta u+Ru$
(interpolating), where $\varepsilon\in\mathbb{R}-\{0\}$, then $P_{\varepsilon
}=2\Delta v+\left\vert \nabla v\right\vert ^{2}+\left(  2\varepsilon+1\right)
R$ satisfies the heat equation%
\begin{align*}
L_{\varepsilon}P_{\varepsilon}  &  =\varepsilon^{-1}\left\vert \nabla\nabla
v\right\vert ^{2}+2\left\langle \operatorname{Rc},\nabla\nabla v\right\rangle
+\varepsilon^{-1}\left\vert \operatorname{Rc}\right\vert ^{2}\\
&  \quad\;+2\varepsilon^{-1}\operatorname{Rc}(\nabla(v-\varepsilon
f),\nabla(v-\varepsilon f))+\left(  1-\varepsilon^{-1}\right)
\operatorname{Rc}\left(  \nabla\left(  v+f\right)  ,\nabla\left(  v+f\right)
\right)  ,
\end{align*}
where $L_{\varepsilon}=\frac{1}{2}(\frac{\partial}{\partial t}-\varepsilon
^{-1}\Delta)-\varepsilon^{-1}\nabla v\cdot\nabla$. The Ricci terms may also be
rewritten as $\left(  1+\varepsilon^{-1}\right)  \operatorname{Rc}%
(\nabla\left(  v-f\right)  ,\nabla\left(  v-f\right)  )+2\left(
\varepsilon-\varepsilon^{-1}\right)  \operatorname{Rc}(\nabla f,\nabla f)$. If
$\varepsilon=-1$, we have%
\[
\left(  \frac{1}{2}(\frac{\partial}{\partial t}+\Delta)+\nabla v\cdot
\nabla\right)  \left(  2\Delta v+\left\vert \nabla v\right\vert ^{2}-R\right)
=-\left\vert \operatorname{Rc}-\nabla\nabla v\right\vert ^{2}=-\left\vert
\nabla\nabla\left(  f+v\right)  \right\vert ^{2}.
\]
This is a special case of Perelman's pointwise energy monotonicity
formula. We interpolated the calculations but not the estimates
between $\varepsilon ^{-1}=-1$ and $\varepsilon^{-1}=1$; for
successful interpolations between Li--Yau--Hamilton inequalities,
see Ni \cite{Ni1}, \cite{Ni05}.

\markright{LINEAR TRACE HARNACK ON A STEADY GRADIENT RICCI SOLITON}

\end{document}